\documentclass{birkjour}
 \newtheorem{thm}{Theorem}[section]
 \newtheorem{cor}[thm]{Corollary}
 \newtheorem{lem}[thm]{Lemma}
 
 \theoremstyle{definition}
 
 \theoremstyle{remark}
 
 \newtheorem{ex}[thm]{Example}
 \numberwithin{equation}{section}

\begin{document}

%
%
%
%
%
%
%
%
%

\author[Youssef Aserrar  and  Elhoucien Elqorachi]{Youssef Aserrar  and  Elhoucien Elqorachi}

\address{%
	Ibn Zohr University, Faculty of sciences, 
Department of mathematics,\\
Agadir,
Morocco}

\email{youssefaserrar05@gmail.com, elqorachi@hotmail.com }



\date{January 1, 2020}
\title[Five trigonometric addition laws on semigroups]{Five trigonometric addition laws on semigroups}
 
\begin{abstract}
In this paper, we determine the complex-valued solutions of the following functional equations 
\[g(x\sigma (y)) = g(x)g(y)+f(x)f(y),\quad x,y\in S,\]
\[f(x\sigma (y)) = f(x)g(y)+f(y)g(x),\quad x,y\in S,\]
 \[f(x\sigma (y)) = f(x)g(y)+f(y)g(x)-g(x)g(y),\quad x,y\in S,\]
\[f(x\sigma(y))=f(x)g(y)+f(y)g(x)+\alpha g(x\sigma(y)),\quad x,y\in S,\]
\[f(x\sigma(y))=f(x)g(y)-f(y)g(x)+\alpha g(x\sigma(y)),\quad x,y\in S,\] 
where $S$ is a semigroup, $\alpha \in \mathbb{C}\backslash \lbrace 0\rbrace$ is a fixed constant and $\sigma :S\rightarrow S$ an involutive automorphism. 
\end{abstract}

\maketitle

\section{Introduction}
Let $S$ be a semigroup and $\sigma$ an involutive automorphism of $S$. That is $\sigma (xy)=\sigma(x)\sigma(y)$ and $\sigma(\sigma(x))=x$ for all $x,y\in S$.\\
The Cosine-Sine functional equation is
\begin{equation}
f(x\sigma (y)) = f(x)g(y)+f(y)g(x)+h(x)h(y),\quad x,y\in S,
\label{cos}
\end{equation}
where $f,g:S\rightarrow\mathbb{C}$.\\
This equation generalizes both the cosine subtraction formula
\begin{equation}
g(x\sigma (y)) = g(x)g(y)+f(x)f(y),\quad x,y\in S,
\label{E3}
\end{equation}
and the sine addition formula 
\begin{equation}
f(x\sigma (y)) = f(x)g(y)+f(y)g(x),\quad x,y\in S.
\label{E1}
\end{equation}
This paper extends previous results about the functional equation  
\begin{equation}
f(x\sigma (y)) = f(x)g(y)+f(y)g(x)-g(x)g(y),\quad x,y\in S,
\label{E2}
\end{equation}
for unknown functions $f,g:S\rightarrow\mathbb{C}$ on semigroups, where $\sigma$ is an involutive automorphism.\\
In a previous paper Ajebbar and Elqorachi \cite{Ajb} gave the general  solution of \eqref{E2} on semigroups generated by its squares, Stetk\ae r \cite{ST2} and Ebanks \cite{EB2,Ebanks} (with $\sigma =Id$) obtained the solutions on semigroups. Here we find the general solution for all semigroups and as an application we describe the solutions of the following functional equation 
\begin{equation}
f(x\sigma(y))=f(x)g(y)+f(y)g(x)+\alpha g(x\sigma(y)),\quad x,y\in S,
\label{App1}
\end{equation}
on semigroups where $\alpha \in \mathbb{C}\backslash \lbrace 0\rbrace$ is a fixed constant.\\
Stetk\ae r \cite[Theorem 3.1]{S} solved the functional equation 
\[f(xy)=f(x)f(y)-g(x)g(y)+\alpha g(xy),\quad x,y\in S,\]
on semigroups, and the solutions of the functional equation 
\begin{equation}
f(x\sigma(y))=f(x)g(y)-f(y)g(x)+\alpha g(x\sigma(y)),\quad x,y\in S,
\label{App2}
\end{equation}
were given by Zeglami et al. \cite[Proposition 4.1]{Z} on topological groups. We extend the results to semigroups, and we relate \eqref{App2} to the functional equation 
\[f(x\sigma(y))=f(x)g(y)-f(y)g(x),\quad x,y\in S,\]
which are solved by Ebanks \cite[Corollary 4.3]{EB1} on monoids, and by the authors \cite[Proposition 3.2]{Ase} on semigroups.\\
The functional equations \eqref{E3} and \eqref{E1} have been investigated by many authors, beginning with the case $S=\left(\mathbb{R},+ \right) $, $\sigma(x)=-x$. Equation \eqref{E3} with $\sigma =Id$ was solved on abelian groups by Vincze \cite{V}, and on general groups by Chung, Kannappan, and Ng \cite{Ch}.\\

The solutions of \eqref{E3} and \eqref{E1} are also described on topological groups by Poulsen and Stetk\ae r \cite{Pou}. Ajebbar and Elqorachi \cite{Ajb,Ajb2} give the solutions of \eqref{E3} and \eqref{E1} on semigroups generated by their squares. A more general description of the solutions of \eqref{E1} with $\sigma =Id$ was obtained recently by Ebanks \cite[Theorem 2.1]{EB1} and \cite[Theorem 3.1]{EB2} on a semigroup not necessarily generated by its squares. Ebanks \cite[Theorem 4.1]{EB1} gives the solution of \eqref{E3} on monoids.\\
Equation \eqref{cos} (with $\sigma=Id$) was solved by Chung, Kannappan, and Ng \cite{Ch} for the case that $S$ is a group. Their results was extended by Ajebbar and Elqorachi \cite{Ajb} to the case that $S$ is a semigroup generated by its squares. The solutions of the special case $h=ig$, $\sigma=Id$, namely 
\begin{equation}
f(xy)=f(x)g(y)+f(y)g(x)-g(x)g(y),\quad x,y\in S,
\label{cog}
\end{equation}
are described by Stetk\ae r \cite{ST2} in terms of multiplicative functions and solutions of the sine addition formula 
\[f(xy)=f(x)g(y)+f(y)g(x),\quad x,y\in S.\]
Using the result obtained by Stetk\ae r \cite{ST2}, Ebanks \cite{Ebanks} gave a description of the solutions of \eqref{cog} on a general semigroup.\\
The purpose of the present paper is to solve the functional equations \eqref{E3}, \eqref{E1}, \eqref{E2} and \eqref{App2} on a semigroup $S$, where $\sigma :S\rightarrow S$ is an involutive automorphism. As an application of \eqref{E2} we solve the functional equation \eqref{App1}.
\section{Our results}
The functional equations \eqref{E3} and \eqref{E1} have not been solved on general semigroups. The present paper does so. We derive explicit formulas for the solutions of \eqref{E3} and \eqref{E1} on semigroups in terms of multiplicative, and additive functions.\\
The complete solution of \eqref{E3} and \eqref{E1} is given respectively in sections 4 and 5. In section 6 we give the general solution of \eqref{E2} on semigroups and we apply the results to solve \eqref{App1} in section 7. The paper concludes with the solutions of \eqref{App2} in section 8 and some examples in section 9.
\section{Notations and terminology}
In order to build our results we recall the following notations and notions that will be used throughout the paper. Let $S$ be a semigroup, i.e a set equipped with an associative binary operation. A function $a$ on $S$
 is additive if $a(xy) = a(x) + a(y)$ for all $x, y\in S$. A  multiplicative function on $S$ is a function $\mu :S\rightarrow \mathbb{C}$ satisfying $\mu(xy) = \mu(x)\mu(y)$ for all $x, y \in S$.\\ 
A function $f:S\rightarrow \mathbb{C}$ is central if $f(xy) = f(yx)$ for all $x, y\in S$, and $f$ is abelian if $f$ is central and $f(xyz)=f(xzy)$ for all $x,y,z\in S$.\\
Let $S$ be a semigroup, $\sigma: S \rightarrow S$ an automorphism. For a subset $T\subseteq S$ we define $T^2:=\lbrace xy\quad \vert \quad x, y\in T\rbrace$.
If $\chi: S \rightarrow \mathbb{C}$ is a multiplicative function and $\chi \neq 0$, we define  the sets
 $$I_{\chi}:=\{x \in S \mid \chi(x)=0\},$$ 
  $$P_\chi : =\{p\in I_{\chi}\backslash I_{\chi}^2\quad \vert up, pv, upv\in I_{\chi}\backslash I_{\chi}^2\quad \text{for all}\quad u,v\in S\backslash I_{\chi}\}.$$   
  For any function $f :S \rightarrow \mathbb{C}$ , we define the functions $f^*=f\circ \sigma$, $f^e=\dfrac{f+f^*}{2}$ and $f^{\circ}=\dfrac{f-f^*}{2}$. We will say that $f$ is even if $f^*=f$ and $f$ is odd if $f^*=-f$. If $g, f: S \rightarrow \mathbb{C}$ are two functions we define the function $(f \otimes g)(x, y):=f(x) g(y)$, for all $x, y \in S$. For a topological semigroup $S$ let $C(S)$ denote the algebra of continuous functions from $S$ into $\mathbb{C}$.\\
  The following lemma will be used throughout the paper.
  \begin{lem}
  Let $\sigma :S\rightarrow S$ be an automorphism and $\chi :S\rightarrow \mathbb{C}$ be a multiplicative function such that $\chi^*=\chi$. Then \\
  $\sigma \left( I_{\chi}\right)= I_{\chi},\ \sigma \left(S\backslash I_{\chi} \right)=S\backslash I_{\chi} ,\ \sigma \left(I_{\chi}\backslash I_{\chi}^2 \right)=I_{\chi}\backslash I_{\chi}^2,\ \sigma \left(P_{\chi} \right)=P_{\chi} $ and $\sigma \left(I_{\chi}\backslash P_{\chi} \right)=I_{\chi}\backslash P_{\chi} $.
  \label{invar}
\end{lem}
\begin{proof}
Simple and elementary considerations, using that $\sigma$ is a bijection. (See \cite[Lemma 4.1]{EB}).
\end{proof}
\section{Solutions of the cosine subtraction formula \eqref{E3}}
The most recent result on the cosine subtraction formula \eqref{E3} on monoids is \cite[Theorem 4.1]{EB2}. In this section we will solve \eqref{E3} on general semigroups.\\
The following lemma will be used later.
\begin{lem}
Let $f:S\rightarrow \mathbb{C}$ be a non-zero function satisfying
\begin{equation}
f(x\sigma(y))=\beta f(x)f(y),\quad\text{for all}\quad x,y\in S,
\label{M1}
\end{equation} 
where $\beta \in \mathbb{C}\backslash \lbrace 0\rbrace$ is a constant. Then there exists a non-zero multiplicative function $\chi:S\rightarrow \mathbb{C}$ such that $\beta f =\chi$ and $\chi^*=\chi$.
\label{M}
\end{lem}
\begin{proof}
By using the associativity of the semigroup operation we compute $f(x\sigma(y)\sigma(z))$ using the identity \eqref{M1} first as $f((x\sigma(y))\sigma(z))$ and then as $f(x(\sigma(y)\sigma(z)))$ and compare the results to obtain 
\begin{equation}
\beta^2 f(x)f(y)f(z)=\beta f(x)f(yz),\quad \text{for all}\quad x,y,z\in S.
\label{M2}
\end{equation}
Since $f\neq 0$ and $\beta \neq 0$ Equation \eqref{M2} can be written as 
\begin{equation}
f(yz)=\beta f(y)f(z), \quad \text{for all}\quad y,z\in S.
\label{M3}
\end{equation}
This implies that the function $\chi:=\beta f$ is multiplicative. On the other hand 
\begin{equation}
f(yz)=f(y\sigma(\sigma(z)))=\beta f(y)f^*(z),\quad \text{for all}\quad y,z\in S.
\label{M4}
\end{equation}
Since $f\neq 0$ and $\beta \neq 0$, we get by comparing equation \eqref{M3} and \eqref{M4} that $f^*=f$, then $\chi^*=\chi$. This completes the proof of Lemma \ref{M}.
\end{proof} 
The next result gives the general solution of \eqref{E3} on semigroups.
\begin{thm}
The solutions $f,g:S\rightarrow\mathbb{C}$ of the functional equation \eqref{E3} can be listed as follows :
\begin{enumerate}
\item[(1)] $g=0$ and $f=0$.
\item[(2)] $g$ is any non-zero function such that $g=0$ on $S^2$, and $f=cg$, where $c\in \lbrace i,-i\rbrace$.
\item[(3)] $g=\dfrac{1}{1+\alpha^2}\chi$ and $f=\dfrac{\alpha}{1+\alpha^2}\chi$, where $\alpha \in \mathbb{C}\backslash \lbrace i,-i\rbrace$ is a constant and $\chi :S\rightarrow\mathbb{C}$ is a non-zero multiplicative function such that $\chi^*=\chi$.
\item[(4)] $g=\dfrac{\delta ^{-1}\chi_1+\delta \chi_2}{\delta^{-1}+\delta}$ and $f=\dfrac{\chi_2-\chi_1}{\delta^{-1}+\delta}$, where $\delta \in \mathbb{C}\backslash \lbrace 0,i,-i\rbrace$  and $\chi_1,\chi_2:S\rightarrow\mathbb{C}$ are two different multiplicative functions such that $\chi_1^* = \chi_1$ and $\chi_2^*=\chi_2$.
\item[(5)] $$
\left\{\begin{array}{l}
f=-i\chi A\quad\text { on } \quad S \backslash I_{\chi} \\
f=0\quad\quad\quad\text { on } \quad I_\chi \backslash P_{\chi} \\
f=-i\rho\quad\quad\text { on } \quad P_\chi  
\end{array}\right.\quad \text{and}\quad g=\chi \pm if,$$
where $\chi:S\rightarrow\mathbb{C}$ is a non-zero  multiplicative function and $A:S\backslash I_{\chi}\rightarrow\mathbb{C}$ an additive function such that $\chi^*=\chi$, $A\circ \sigma =A$ and $\rho :P_{\chi}\rightarrow\mathbb{C}$ is an even function. In addition we have the following conditions :\\
(I): If $x\in\left\lbrace up,pv,upv \right\rbrace $ for $p\in P_{\chi}$ and $u,v\in S\backslash I_{\chi}$, then $x\in P_{\chi}$ and we have respectively $\rho(x)=\rho(p)\chi (u)$, $\rho(x)=\rho(p)\chi (v)$, or $\rho(x)=\rho(p)\chi (uv)$.\\
(II): $f(xy)=f(yx)=0$ for all $x\in I_{\chi}\backslash P_{\chi}$ and $y\in S\backslash I_{\chi}$.\\
Note that $f$ and $g$ are Abelian in each case.\par
Furthermore, if $S$ is a topological semigroup and $g\in C(S)$, then\\ $f,\chi,\chi_1,\chi_2\in C(S)$,  $A\in C(S\backslash I_{\chi})$ and $\rho \in C(P_{\chi})$.
\end{enumerate}
\label{TE3}
\end{thm}
\begin{proof}
If $g=0$ then $f=0$. This is case (1). So from now on we assume that $g\neq 0$. Suppose that $g=0$ on $S^2$. For all $x,y\in S$ we have $x\sigma(y)\in S^2$, so we get from equation \eqref{E3} that 
\begin{equation}
g(x)g(y)+f(x)f(y)=0,
\label{A1}
\end{equation}
since $g\neq 0$ we obtain from equation \eqref{A1} that $f=cg$ where $c\in \mathbb{C}$ is a constant, then if we take this into account in equation \eqref{A1} we get \\$(c^2+1)g(x)g(y)=0$, this implies that $c^2+1=0$ because $g\neq 0$, so $c\in \lbrace i,-i\rbrace$. This occurs in part (2) of Theorem \ref{TE3}. If $f=0$, then equation \eqref{E3} can be written as follows $$g(x\sigma(y))=g(x)g(y),\quad x,y\in S.$$ According to Lemma \ref{M}, $g=:\chi$ is multiplicative and $\chi^*=\chi$. This occurs in part (3) of Theorem \ref{TE3} with $\alpha=0$.\\
 Now we assume that $g\neq 0$ on $S^2$, $f\neq 0$ and we  discuss two cases according to whether $f$ and  $g$ are linearly dependent or not.\\
\underline{First case :} $g$ and $f$ are linearly dependent. There exists a constant $\alpha \in \mathbb{C}$ such that $f=\alpha g$, so equation \eqref{E3} can be written as 
\begin{equation}
g(x\sigma(y))=(1+\alpha^2) g(x)g(y),\quad \text{for all}\quad x,y\in S.
\label{A3}
\end{equation}
Since $g\neq 0$ on $S^2$ and $f\neq 0$, we deduce from \eqref{A3} that $\alpha \notin \lbrace 0,i,-i\rbrace$, and then according to Lemma \ref{M}, $\chi:=(1+\alpha^2) g$ is multiplicative and $\chi^*=\chi$. This occurs in case (3) with $\alpha \neq 0$.\\
\underline{Second case :} $g$ and $f$ are linearly independent. By using the associativity of the semigroup operation we compute $g(x\sigma(y)\sigma(z))$ by the help of equation \eqref{E3} first as $g((x\sigma(y))\sigma(z))$ and then as $g(x(\sigma(y)\sigma(z)))$ and compare the results. We obtain after some rearrangement that
\begin{equation}
f(x)\left[f(yz)-f(y)g(z) \right]+g(x)\left[ g(yz)-g(y)g(z)\right]=f(z)f(x\sigma(y)).
\label{A4}  
\end{equation}
Since $f\neq 0$, there exists $z_0\in S$ such that $f(z_0)\neq 0$ and then 
\begin{equation}
f(x)h(y)+g(x)k(y)=f(x\sigma(y)),
\label{A5}
\end{equation}
where 
$$h(y)=\dfrac{f(yz_0)-f(y)g(z_0)}{f(z_0)},$$
and $$k(y)=\dfrac{g(yz_0)-g(y)g(z_0)}{f(z_0)}.$$
By using \eqref{E3} and $\sigma(\sigma(y))=y$ for all $y\in S$, $k(y)$ can be written as follows 
\begin{equation}
k(y)=c_1g(y)+c_2f(y),\quad\text{for all}\quad y\in S,
\label{A6}
\end{equation}
where $c_1=\dfrac{g(\sigma (z_0))-g(z_0)}{f(z_0)}$ and $c_2=\dfrac{f(\sigma(z_0))}{f(z_0)}$. Substituting  \eqref{A5} into \eqref{A4}, we obtain 
\begin{align}
\begin{split}
f(x)\left[f(yz)-f(y)g(z) \right]+g(x)\left[ g(yz)-g(y)g(z)\right]\\= f(x)f(z)h(y)+g(x)f(z)k(y).
\label{A7}
\end{split}
\end{align}
Since $f$ and $g$ are linearly independent we deduce from \eqref{A7} that
\begin{equation}
g(yz)=g(y)g(z)+f(z)k(y),
\label{A8}
\end{equation}
and 
\begin{equation}
f(yz)=f(y)g(z)+f(z)h(y).
\label{A9}
\end{equation}
Substituting \eqref{A6} into \eqref{A8}, we get
\begin{equation}
g(yz)=\left[ g(z)+c_1f(z)\right] g(y)+c_2f(z)f(y).
\label{A10}
\end{equation}
On the other hand, by applying \eqref{E3} to the pair $(y,\sigma (z))$ and taking into account that $\sigma(\sigma(z))=z$ for all $z\in S$, we obtain 
\begin{equation}
g(yz)=g(y)g^*(z)+f(y)f^*(z).
\label{A11}
\end{equation}
By comparing \eqref{A10} and \eqref{A11}, and using the linear independence of $f$ and $g$, we get 
\begin{equation}
g^*=g+c_1f,
\label{A12}
\end{equation}
\begin{equation}
f^*=c_2f.
\label{A13}
\end{equation}
Since $f\neq 0$, we deduce from \eqref{A13} that $c_2^2=1$. So either $f=f^*$ or $f=-f^*$.\\
\underline{Subcase A :} $f=f^*$. In view of \eqref{A12}, we have 
$$g=g^*+c_1f^*=g+c_1f+c_1f=g+2c_1f.$$
This implies that $2c_1f=0$, and then $c_1=0$ since $f\neq 0$. So $g^*=g$ and equation \eqref{E3} can be written as follows \begin{equation}
g(xy)=g(x)g(y)+f(x)f(y),\quad x,y\in S.
\label{A14}
\end{equation}
Defining $l:=if$, equation \eqref{A14} can be written as follows
$$g(xy)=g(x)g(y)-l(x)l(y),\quad x,y\in S.$$
According to \cite[Theorem 3.2]{EB1} and taking into account its improvement in \cite[Theorem 3.1]{EB2}, and that $f,g$ are linearly independent, we have the following possibilities :\\
(i) $g=\dfrac{\delta ^{-1}\chi_1+\delta \chi_2}{\delta^{-1}+\delta}$ and $l=\dfrac{\chi_1-\chi_2}{i(\delta^{-1}+\delta)}$, where $\delta \in \mathbb{C}\backslash \lbrace 0,i,-i\rbrace$ is a constant and $\chi_1,\chi_2:S\rightarrow\mathbb{C}$ are two multiplicative functions such that $\chi_1\neq \chi_2$. Since $g=g^*$, $f=f^*$ and $l=if$, we deduce that $f=\dfrac{\chi_2-\chi_1}{\delta^{-1}+\delta}$, $\chi_1=\chi_1^*$ and $\chi_2=\chi_2^*$. This is case (4).\\
(ii) $g=\chi \pm l$ and  $
\left\{\begin{array}{l}
l=\chi A\quad\text { on } \quad S \backslash I_{\chi} \\
l=0\quad\quad\text { on } \quad I_\chi \backslash P_{\chi} \\
l=\rho\quad\quad\text { on } \quad P_\chi 
\end{array}\right.
$, where $\chi:S\rightarrow\mathbb{C}$ is a non-zero  multiplicative function and $A:S\backslash I_{\chi}\rightarrow\mathbb{C}$ is an additive function and $\rho :P_{\chi}\rightarrow\mathbb{C}$ is a function with conditions (I)(for $\rho$) and (II)(for $l$) holding. This implies that 
$$
\left\{\begin{array}{l}
f=-i\chi A\quad\text { on } \quad S \backslash I_{\chi} \\
f=0\quad\quad\quad\text { on } \quad I_\chi \backslash P_{\chi} \\
f=-i\rho\quad\quad\text { on } \quad P_\chi  
\end{array}\right.\quad \text{and}\quad g=\chi \pm if.$$
Since $f^*=f$ and $g^*=g$, we see that  $\chi^*=\chi$, then by using Lemma \ref{invar} that $\rho \circ \sigma=\rho$ and $A\circ \sigma =A$. In addition $l=if$ implies that $f$ satisfies the condition (II). This occurs in part (5).\\
\underline{Subcase B :} $f^*=-f$.\\
\underline{Subcase B.1 :} $c_1=0$. So $g^*=g$ and equation \eqref{E3} can be written as follows
$$g(xy)=g(x)g(y)-f(x)f(y).$$
Similarly to the previous case, we get according to \cite[Theorem 3.2]{EB1} and taking into account its improvement in \cite[Theorem 3.1]{EB2} and that $f$ and $g$ are linearly independent, the two cases:\\
(i) $g=\dfrac{\delta ^{-1}\chi_1+\delta \chi_2}{\delta^{-1}+\delta}$ and $f=\dfrac{\chi_1-\chi_2}{i(\delta^{-1}+\delta)}$, where $\delta \in \mathbb{C}\backslash \lbrace 0,i,-i\rbrace$ is a constant and $\chi_1,\chi_2:S\rightarrow\mathbb{C}$ are two different multiplicative functions. Since $f^*=-f$ and $g^*=g$, we get 
\begin{equation}
\delta ^{-1}(\chi_1-\chi_1^*)+\delta (\chi_2-\chi_2^*)=0,
\label{Par1}
\end{equation}
\begin{equation}
\chi_1+\chi_1^*=\chi_2+\chi_2^*.
\label{Par2}
\end{equation}
Since $\chi_1\neq\chi_2$, we obtain by the help of \cite[Corollary 3.19]{ST1} that $\chi_1=\chi_2^*$. Then \eqref{Par1} reduces to $\chi_1=\chi_2$. This case does not occur.\\
(ii) $g=\chi \pm f$ and  $
\left\{\begin{array}{l}
f=\chi A\quad\text { on } \quad S \backslash I_{\chi} \\
f=0\quad\quad\text { on } \quad I_\chi \backslash P_{\chi} \\
f=\rho\quad\quad\text { on } \quad P_\chi  
\end{array},\right.
$ where $\chi:S\rightarrow\mathbb{C}$ is a non-zero multiplicative function, $A:S\backslash I_{\chi}\rightarrow\mathbb{C}$ is a non-zero additive function and $\rho :P_{\chi}\rightarrow\mathbb{C}$ is a non-zero function with conditions (I) and (II) holding. If $g=\chi+ f$ we obtain since $f^*=-f$ and $g^*=g$ that $g=\chi^*- f$. Adding and subtracting this from $g=\chi +f$ we get that 
\[g=\dfrac{\chi+\chi^*}{2}\quad\text{and}\quad f=\dfrac{\chi^*-\chi}{2}.\]
By assumption $f\neq 0$, so $\chi\neq \chi^*$. This shows that we deal with case 4 for $\delta =1$. Now if $g=\chi-f$, we show by the same way that $g=\dfrac{\chi+\chi^*}{2}$ and $f=\dfrac{\chi-\chi^*}{2}$ which occurs in case 4 with $\delta =1$.\\
\underline{Subcase B.2 :} $c_1\neq 0$. From \eqref{A12} we have $g^*(xy)=g(xy)+c_1f(xy)$ for all $x,y\in S$, and then by using \eqref{E3} and \eqref{A9}, we get the identity 
\begin{align*}
g^*(x)g(y)+f^*(x)f(y)=g(x)g^*(y)+f(x)f^*(y)\\
+c_1\left[f(x)g(y)+f(y)h(x) \right].
\end{align*}

Since $f^*=-f$,  $g^*=g+c_1f$, $c_1\neq 0$ and $f\neq 0$, the identity above reduces to $h=-g$. Then equation \eqref{A9} becomes
\begin{equation}
f(yz)=f(y)g(z)-f(z)g(y),\quad\text{for all}\quad y,z\in S.
\label{A15}
\end{equation}
Using the associativity of the semigroup operation and \eqref{E3}, we get from \eqref{A15} that 
\begin{align*}
f(x)g(yz)-g(x)f(yz)=f(x)\left[g(y)g(z)+f(y)f(z) \right]\\ -g(x)\left[f(y)g(z)+f(z)g^*(y) \right],
\end{align*}

and then by the linear independence of $f$ and $g$, we deduce that
\begin{equation}
f(yz)=f(y)g(z)+f(z)g^*(y).
\label{A16}
\end{equation}
Comparing \eqref{A15} and \eqref{A16}, and using the linear independence of $f$ and $g$, we get that $g^*=-g$, then \eqref{A12} becomes $-g=g+c_1f$, so $2g+c_1f=0$. This contradicts the fact that $f$ and $g$ are linearly independent, so this case does not occur.\par
For the converse it is easily checked that the forms (1), (2), (3), (4) and (5) satisfy \eqref{E3}.\par
Finally, suppose that $S$ is a topological semigroup and $g\in C(S)$. In case (1) there is nothing to prove. Now if $f\neq 0$, the continuity of $f$ follows easily from the continuity of $g$ and the functional equation \eqref{E3}: Choose $y_0\in S$ such that $f(y_0)\neq 0$, we get from \eqref{E3} that 
\[f(x)=\dfrac{g(x\sigma(y_0))-g(y_0)g(x)}{f(y_0)}\ \text{for}\ x\in S.\] 
The function $x\mapsto g(x\sigma(y_0))$ is continuous, since the right translation $x\mapsto x\sigma(y_0)$ from $S$ into $S$ is continuous.\\
In case (5), the functions $\rho$ and $\chi A$ inherits continuity from $f$ by restriction. So $A$ is continuous since $\chi$ is non-zero. In case (4) we get the continuity of $\chi_1$ and $\chi_2$ by the help of \cite[Theorem 3.18]{ST1}. This completes the proof of Theorem \ref{TE3}.
\end{proof}
\section{Solutions of the sine addition formula \eqref{E1}}
The solution of the sine subtraction formula 
\[g(x\sigma(y))=g(x)f(y)-g(y)f(x),\quad x,y\in S,\]
 on a general monoid was given recently by Ebanks \cite[Corollary 4.3]{EB2}. Here we find the solutions of the sine addition formula \eqref{E1} on semigroups.
\begin{thm}
The solutions $f,g : S\rightarrow \mathbb{C}$ of Equation \eqref{E1} are the following pairs
\begin{enumerate}
\item[(1)] $f=0$ and $g$ is arbitrary.
\item[(2)] $f$ is any non-zero function such that $f=0$ on $S^2$, while $g=0$.
\item[(3)] $f=\dfrac{1}{2\alpha}\chi$ and $g=\dfrac{1}{2}\chi$, where $\chi :S\rightarrow \mathbb{C}$ is a non-zero multiplicative function such that $\chi^*=\chi$ and $\alpha\in \mathbb{C}\backslash \lbrace 0\rbrace$.
\item[(4)] $f=c\left(\chi_1-\chi_2 \right) $ and $g=\dfrac{\chi_1+\chi_2}{2}$, where $\chi_1,\chi_2 :S\rightarrow \mathbb{C}$ are two different multiplicative functions such that $\chi_1^*=\chi_1$, $\chi_2^*=\chi_2$ and $c\in \mathbb{C}\backslash \lbrace 0\rbrace$.
\item[(5)] $$
\left\{\begin{array}{l}
f=\chi A \quad \text { on } \quad S \backslash I_{\chi} \\
f=0\quad\quad \text { on } \quad I_\chi \backslash P_{\chi} \\
f=\rho  \quad\quad \text { on } \quad P_\chi  
\end{array}\right.\quad \text{and}\quad g=\chi,
$$
where  $\chi: S \rightarrow \mathbb{C}$ is a non-zero multiplicative function and $A: S \backslash I_{\chi} \rightarrow \mathbb{C}$ is a non-zero additive function such that $\chi^{*}=\chi$, $A \circ \sigma=A$, and $\rho: P_{\chi} \rightarrow \mathbb{C}$ is an even function. In addition we have the following conditions : \newline
(I): If $x\in\left\lbrace up,pv,upv \right\rbrace $ for $p\in P_{\chi}$ and $u,v\in S\backslash I_{\chi}$, then $x\in P_{\chi}$ and we have respectively $\rho(x)=\rho(p)\chi (u)$, $\rho(x)=\rho(p)\chi (v)$, or $\rho(x)=\rho(p)\chi (uv)$.\\
(II): $f(xy)=f(yx)=0$ for all $x\in S\backslash I_{\chi}$ and $y\in I_{\chi}\backslash P_{\chi}$.\\
Note that, off the exceptional case (1) $f$ and $g$ are Abelian.\par
Furthermore, off the exceptional case (1),  if $S$ is a topological semigroup and $f\in C(S)$, then $g,\chi$,$\chi_1,\chi_2\in C(S)$,  $A\in C(S\backslash I_{\chi})$ and $\rho \in C(P_{\chi})$.
\end{enumerate}
\label{P1}
\end{thm}
\begin{proof}
If $f=0$ then $g$ will be arbitrary. This occurs in case (1). From now on we assume that $f\neq 0$. Suppose that $f=0$ on $S^2$. For all $x,y\in S$, we get from equation \eqref{E1} that 
\begin{equation}
f(x)g(y)+f(y)g(x)=0,
\label{B1}
\end{equation}
since $f\neq 0$ we obtain from equation \eqref{B1} that $g=cf$ where $c\in \mathbb{C}$ is a constant, then if we take this into account in equation \eqref{B1} we get $2cf(x)f(y)=0$, for all $x,y\in S$. This implies that $c=0$ because $f\neq 0$. This occurs in part (2) of Theorem \ref{P1}. Now we assume that $f\neq 0$ on $S^2$ and we  discuss two cases according to whether $f$ and  $g$ are linearly dependent or not.\\
\underline{First case :} $f$ and $g$ are linearly dependent. There exists a constant $\alpha \in \mathbb{C}$ such that $g=\alpha f$, so equation \eqref{E1} can be written as follows
\begin{equation}
f(x\sigma(y))=2\alpha f(x)f(y),\quad \text{for all}\quad x,y\in S.
\label{B2}
\end{equation}
Since $f\neq 0$ on $S^2$, then $\alpha \neq 0$. According to Lemma \ref{M}, the function $\chi:=2\alpha f$ is multiplicative and $\chi^*=\chi$. This is case (3).\\
\underline{Second case :} $f$ and $g$ are linearly independent. By using the associativity of the semigroup operation we compute $f(x\sigma(y)\sigma(z))$ using equation \eqref{E1} first as $f((x\sigma(y))\sigma(z))$ and then as $f(x(\sigma(y)\sigma(z)))$ and compare the results. We obtain after some rearrangement that
\begin{equation}
f(x)\left[g(yz)-g(y)g(z) \right]+g(x)\left[ f(yz)-f(y)g(z)\right]=f(z)g(x\sigma(y)).
\label{B6}  
\end{equation}
Since $f\neq 0$, there exists $z_0\in S$ such that $f(z_0)\neq 0$ and hence 
\begin{equation}
f(x)h(y)+g(x)k(y)=g(x\sigma(y)),\quad x,y\in S,
\label{B7}
\end{equation}
where 
$$h(y)=\dfrac{g(yz_0)-g(y)g(z_0)}{f(z_0)},$$
and $$k(y)=\dfrac{f(yz_0)-f(y)g(z_0)}{f(z_0)}.$$
By using \eqref{B7}, equation \eqref{B6} becomes
\begin{align}
\begin{split}
f(x)\left[g(yz)-g(y)g(z) \right]+g(x)\left[ f(yz)-f(y)g(z)\right]\\= f(x)f(z)h(y)+g(x)f(z)k(y),\quad x,y,z\in S.
\label{B8}
\end{split}
\end{align}
Since $f$ and $g$ are linearly independent we deduce from \eqref{B8} that
\begin{equation}
g(yz)=g(y)g(z)+f(z)h(y),
\label{D1}
\end{equation}
and 
\begin{equation}
f(yz)=f(y)g(z)+f(z)k(y),\quad \text{for all}\quad y,z\in S.
\label{B9}
\end{equation}
Equation \eqref{E1} implies that $f(yz_0)=f(y)g^*(z_0)+f^*(z_0)g(y)$ for all $y\in S$, so
\begin{equation}
k(y)=\alpha f(y)+\beta g(y),
\label{B10}
\end{equation}
where 
\[\alpha =\dfrac{g^*(z_0)-g(z_0)}{f(z_0)}\quad \text{and}\quad \beta =\dfrac{f^*(z_0)}{f(z_0)}.\]
By using \eqref{B10}, equation \eqref{B9} can be written as follows
$$f(yz)=\left(g(z)+\alpha f(z) \right)f(y)+\beta f(z)g(y),\quad y,z\in S.$$
On the other hand equation \eqref{E1} implies that 
$$f(yz)=f(y)g^*(z)+f^*(z)g(y),\quad y,z\in S.$$
By comparing these last two identities it follows from the linear independence of $f$ and $g$ that 
\begin{equation}
g^*=g+\alpha f
\label{B11}
\end{equation}
\begin{equation}
f^*=\beta f.
\label{B12}
\end{equation}
Since $f\neq 0$ we get from \eqref{B12} that $\beta \neq 0$ and $\beta ^2=1$.\\ 
On the other hand for all $x,y\in S$ we have 

\begin{align*}
f(x\sigma(y)) &= \beta^{-1}f^*(x\sigma(y))=\beta^{-1}f(\sigma(x)\sigma(\sigma(y))) \\
& = \beta^{-1}f^*(x)g^*(y)+\beta^{-1}f^*(y)g^*(x) \\
& = f(x)\left[g(y)+\alpha f(y) \right]+f(y)\left[g(x)+\alpha f(x) \right] \\
&=f(x)g(y)+f(y)g(x)+2\alpha f(x)f(y)\\
&=f(x\sigma(y))+2\alpha f(x)f(y). 
\end{align*}
So $\alpha=0$ since $f\neq 0$, and then $g^*=g$.\\
\underline{Subcase A :} $\beta =-1$. In this case $f^*=-f$ and equation \eqref{E1} becomes 
\begin{equation}
f(xy)=f(x)g(y)-f(y)g(x),\quad x,y\in S.
\label{B15}
\end{equation}
By using the associativity of the semigroup operation and taking into account equation \eqref{D1}, we compute $f(xyz)$ using equation \eqref{B15} first as $f(x(yz))$ and then as $f((xy)z)$ and compare the results. We obtain after some simplification that 
\begin{equation}
2g(x)g(y)=-f(x)h(y)-f(y)h(x),\quad x,y\in S.
\label{B16}
\end{equation}
Since $f\neq 0$, equation \eqref{B16} implies that $h=af+bg$ for some constants $a,b\in \mathbb{C}$, taking this into account in \eqref{B16} we obtain 
\begin{equation}
2g(x)g(y)=\left( -2af(y)-bg(y)\right) f(x)-bg(x)f(y).
\label{B17}
\end{equation}
By using the linear independence of $f$ and $g$, we deduce from \eqref{B17} that
\begin{equation}
2g+bf=0.
\label{B18}
\end{equation}
Equation \eqref{B18} contradicts the fact that $f$ and $g$ are linearly independent. This case does not occur.\\
\underline{Subcase B :} $\beta =1$. Since $f^*=f$ and equation \eqref{E1} becomes 
\begin{equation}
f(xy)=f(x)g(y)+f(y)g(x).
\end{equation}
Then according to \cite[Theorem 3.1]{EB2} and taking into account that $f\neq 0 $, $g\neq 0$, $f^*=f$ and $ g^*=g$ we have the following possibilities :\\
(i) $f=c\left( \chi _1 -\chi _2\right) $ and $g=\dfrac{\chi _1+\chi _2}{2}$, for some constant $c\in \mathbb{C}\backslash \lbrace 0\rbrace$ and $\chi_1, \chi _2 : S\rightarrow \mathbb{C}$ are two  multiplicative functions such that $\chi _1 \neq \chi _2$, $\chi_1^*=\chi_1$ and $\chi_2^*=\chi_2$. This is case (4).\\
(ii) $$
\left\{\begin{array}{l}
f=\chi A \quad \text { on } \quad S \backslash I_{\chi} \\
f=0 \quad\quad \text { on } \quad I_\chi \backslash P_{\chi} \\
f=\rho\quad \quad \text { on } \quad P_\chi  
\end{array}\right.\quad\text{and}\quad g=\chi,
$$
where  $\chi: S \rightarrow \mathbb{C}$ is a non-zero multiplicative function and $A: S \backslash I_{\chi} \rightarrow \mathbb{C}$ is a non-zero additive function such that $\chi^*=\chi, A\circ \sigma =A$, $\rho: P_{\chi} \rightarrow \mathbb{C}$ is an even function with conditions (I) and (II) holding. This occurs in part (5) of Theorem \ref{P1}.\par 
 Conversely we check by elementary computations that if $f, g$ have one of the forms (1)--(5) then $(f,g)$ is a solution of equation \eqref{E1}.\par
For the continuity statements, the continuity of $g$ follows easily from the continuity of $f$ and the functional equation \eqref{E1}. Let $y_0\in S$ such that $f(y_0)\neq 0$, in view of \eqref{E1} we have
\[g(x)=\dfrac{f(x\sigma(y_0))-g(y_0)f(x)}{f(y_0)},\quad x\in S.\]
The function $x\mapsto f(x\sigma(y_0))$ is continuous, since $S$ is a topological semigroup so that the right translation $x\mapsto x\sigma(y_0)$ from $S$ into $S$ is continuous.\\
In case (5), the functions $\rho$ and $\chi A$ inherits continuity from $f$ by restriction. So $A$ is continuous since $\chi$ is non-zero. In case (4) we get the continuity of $\chi_1$ and $\chi_2$ by the help of \cite[Theorem 3.18]{ST1}.
 This completes the proof of Theorem \ref{P1}.
\end{proof}
\section{Solutions of Equation \eqref{E2}}
In the two following lemmas, we give some key properties of the solutions of Equation \eqref{E2}.
\begin{lem}
Let $f, g : S\rightarrow \mathbb{C}$  be a solution of the functional equation \eqref{E2}. Then
\begin{enumerate}
\item[(1)] $f(xy)=f^*(yx)$, for all $x, y\in S$.
\item[(2)] $f(xyz)=f^*(xyz)$, for all $x, y, z\in S$.
\end{enumerate}
\label{Lem2}
\end{lem}
\begin{proof}
\begin{enumerate}
\item[(1)] The right hand side of the functional equation \eqref{E2} is invariant under the interchange of $x$ and $y$. So $f(x\sigma(y))=f(y\sigma(x))$ for all $x, y$\\$\in S$, then $f(xy)=f(x\sigma (\sigma(y))=f(\sigma(y)\sigma(x))=f^*(yx)$.
\item[(2)] For all $x, y, z\in S$ we have
$$f(xyz)=f^*(zxy)=f(yzx)=f^*(xyz).$$ This completes the proof of Lemma \ref{L1}.
\end{enumerate}
\end{proof}
\begin{lem}
Let $f, g : S\rightarrow \mathbb{C}$  be a solution of the functional equation \eqref{E2}. Then for all $x,y,z\in S$, the following statements hold:
\begin{enumerate}
\item[(1)] \begin{multline}
\left(f(x)-g(x) \right) \left(g(yz)-g(y)g(z) \right) =\left(f(z)-g(z) \right)g(x\sigma(y)) \\+g(x)\left(f(y)g(z)-f(yz) \right).
\label{L1} 
\end{multline}
\item[(2)] \begin{multline}
\left(f(x)-g(x) \right) \left(g(yz)-g(y)g(z) \right) =\left(f(z)-g(z) \right)g(x\sigma(y)) \\+g(x)g(y)\left(g^*(z)-f^*(z) \right)+g(x)f(y)\left(g(z)-g^*(z) \right) .
\label{L2} 
\end{multline}
\item[(3)] Suppose that $f$ and $g$ are linearly independent. Then
\begin{enumerate}
\item[(a)]there exists two functions $h, k : S\rightarrow \mathbb{C}$ such that 
\begin{equation}
g(x\sigma(y))=f(x)h(y)+g(x)k(y).
\label{L3}
\end{equation}
\item[(b)] there exists two constants $\alpha ,\beta \in \mathbb{C}$ such that 
$$h+k=\alpha g+\beta f,$$
\begin{equation}
g=(1-\beta )g^*+\beta f^*,
\label{L4}
\end{equation}
\begin{equation}
\alpha (f^*-g^*)=f-g.
\label{L5}
\end{equation}
\item[(c)] $\alpha \in \lbrace -1,1\rbrace$, and $\alpha =1$ implies that $\beta =0$, $g^*=g$ and $f^*=f$.
\end{enumerate} 

\end{enumerate}
\label{Lem1}
\end{lem}
\begin{proof}
(1) By using the associativity of the semigroup operation we compute $f(x\sigma(y)\sigma(z))$ with the help of equation \eqref{E2} first as $f((x\sigma(y))\sigma(z))$ and then as $f(x(\sigma(y)\sigma(z)))$ and compare the results. We obtain after some rearrangement the identity \eqref{L1}.\\
(2) From equation \eqref{E2} we get that 
\begin{equation} 
f(yz)=f(y\sigma(\sigma(z)))=f(y)g^*(z)+f^*(z)g(y)-g(y)g^*(z).
\label{L5}
\end{equation}
Taking into account equation \eqref{L5} in the identity \eqref{L1}, we obtain the identity \eqref{L2}.\\
(3)(a) Since $f$ and $g$ are linearly independent, then $f-g\neq 0$, so there exists $z_0\in S$ such that $f(z_0)-g(z_0)\neq 0$. Putting $z=z_0$ in \eqref{L1} we find that 
\[g(x\sigma(y))=f(x)h(y)+g(x)k(y),\]
for some functions $h$ and $k$. This is case (3)(a).\\
(b) Substituting the identity \eqref{L3} into \eqref{L2}, we get 
\begin{align*}
&f(x)\left[ g(yz)-g(y)g(z)\right]+g(x)\left[ g(y)g(z)-g(yz)\right]\\
&=f(x)\left[ (f(z)-g(z))h(y)\right]\\ &+g(x)\left[(f(z)-g(z))k(y)+g(y)(g^*(z)-f^*(z))+f(y)(g(z)-g^*(z)) \right].   
\end{align*}
By applying the linear independence of $f$ and $g$, we deduce that
\begin{equation}
g(yz)-g(y)g(z)=(f(z)-g(z))h(y),
\label{L6}
\end{equation}
\begin{align}
\begin{split}
&g(yz)-g(y)g(z)=(g(z)-f(z))k(y)\\
&+g(y)(f^*(z)-g^*(z))+f(y)(g^*(z)-g(z)).
\label{L7}
\end{split}
\end{align}
Then we deduce from \eqref{L6} and \eqref{L7} that
\begin{align}
\begin{split}
(f(z)-g(z))(h(y)+k(y))&=g(y)(f^*(z)-g^*(z))\\
&+f(y)(g^*(z)-g(z)).
\label{L8}
\end{split}
\end{align}
Since $f$ and $g$ are linearly independent, then $f-g\neq 0$, so there exists $z_0\in S$ such that $f(z_0)-g(z_0)\neq 0$. Putting $z=z_0$ in \eqref{L8} we find that 
\begin{equation}
h(y)+k(y)=\alpha g(y)+\beta f(y),
\label{v1}
\end{equation}
for some constants $\alpha ,\beta \in  \mathbb{C}$ which is the first part of (3)(b).\\
Sustituting \eqref{L3} into \eqref{L1}, we get
\begin{align*}
&f(x)\left[ g(yz)-g(y)g(z)\right]+g(x)\left[ g(y)g(z)-g(yz)\right]\\
&=f(x)\left[ (f(z)-g(z))h(y)\right]\\
&+g(x)\left[(f(z)-g(z))k(y)-f(yz)+f(y)g(z)\right].   
\end{align*}
By applying the linear independence of $f$ and $g$, we get
\begin{equation}
f(yz)=f(y)g(z)-g(y)g(z)+g(yz)+(f(z)-g(z))k(y).
\label{L10}
\end{equation}
Then we deduce from \eqref{L6} and \eqref{L10} that
\begin{equation}
f(yz)=f(y)g(z)+(f(z)-g(z))(k(y)+h(y)).
\label{L11}
\end{equation}
By using \eqref{v1}, equation \eqref{L11} becomes
$$f(yz)=f(y)g(z)+(f(z)-g(z))(\alpha g(y)+\beta f(y)).$$
Then we deduce that 
\begin{align}
\begin{split}
f(y\sigma(z))&=f(y)\left[(1-\beta) g^*(z)+\beta f^*(z) \right]\\
&+g(y)\left[\alpha (f^*(z)-g^*(z)) \right].
\label{L12}
\end{split}  
\end{align}
On the other hand, equation \eqref{E2} can be written as follows
\begin{equation}
f(y\sigma(z))=f(y)g(z)+g(y)\left[f(z)-g(z) \right]. 
\label{L13}
\end{equation}
By comparing \eqref{L12} and \eqref{L13} and using the linear independence of $f$ and $g$, we get  that $g=(1-\beta)g^*+\beta f^*$ and $\alpha (f^*-g^*)=f-g$. This is case (3)(b).\\
(c) According to (3)(b), $\alpha (f^*-g^*)=f-g$. This implies that $\alpha (f-g)=f^*-g^*$, then $\alpha ^2 (f^*-g^*)=f^*-g^*$. So $\alpha^2=1$ since $f\neq g$. In addition if $\alpha=1$ we have $f^*-g^*=f-g$, and then
\begin{equation}
g=(1-\beta)g^*+\beta f^*=g^*+\beta (f^*-g^*)=g^*+\beta (f-g).
\label{etoi1}
\end{equation}
This implies that 
\begin{equation}
g^*=g+\beta (f-g).
\label{etoi2}
\end{equation}
By adding \eqref{etoi1} to \eqref{etoi2} we find that
 \[g+g^*=g+g^*+2\beta (f-g),\]
 which implies that $2\beta (f-g)=0$. Since $f\neq g$ ($f$ and $g$ are linearly independent) it follows that $\beta =0$. Then $g=g^*$ and $f=f^*$. This is case (3)(c).
 This completes the proof of Lemma \ref{Lem1}.
\end{proof}
Now we are ready to solve the functional equation \eqref{E2}. Like in Stetk\ae r \cite{ST2} and Ebanks \cite{Ebanks} the solutions $\phi$ of a special instance of the sine addition law play an important role, and Stetk\ae r \cite[Theorem 3.3]{ST2} is Theorem \ref{Thm} with $\sigma =Id$.
\begin{thm}
The solutions $f, g : S\rightarrow \mathbb{C}$  of the functional equation \eqref{E2} can be listed as follows :
\begin{enumerate}
\item[(1)] $f=0$ and $g=0$.
\item[(2)] $f$ is any non-zero function such that $f=0$ on $S^2$, and $g=0$.
\item[(3)]  $f$ is any non-zero function such that $f=0$ on $S^2$, and $g=2f$.
\item[(4)] $f=\dfrac{\beta^2}{2\beta -1}\chi$ and $g=\beta\chi$, where $\chi$ is a non-zero even multiplicative function and $\beta  \in \mathbb{C}\backslash \left\lbrace  0, \dfrac{1}{2}\right\rbrace $ is a constant.
\item[(5)] There exist two different, non-zero even multiplicative functions $\chi_1, \chi_2:S\rightarrow\mathbb{C}$ and a constant $c_1\in \mathbb{C}\backslash\lbrace0,-1,1\rbrace$ such that
$$f=\dfrac{\chi_1+\chi_2}{2}+\dfrac{c_1^2+1}{2c_1}\dfrac{\chi_1-\chi_2}{2}\quad\text{and}\quad g=\dfrac{\chi_1+\chi_2}{2}+c_1\dfrac{\chi_1-\chi_2}{2}.$$
\item[(6)] There exist a non-zero even multiplicative function $\chi :S\rightarrow\mathbb{C}$ and a non-zero even function $\phi :S\rightarrow\mathbb{C}$ satisfying 
$$\phi(xy)=\phi(x)\chi(y)+\phi(y)\chi(x),$$
such that 
$$\left\{\begin{array}{l}
f=\dfrac{1}{2}\phi +\chi \\
g=\phi +\chi\end{array}.\right.$$
\item[(7)] There exist a non-zero even multiplicative function $\chi :S\rightarrow\mathbb{C}$ and a non-zero even function $\phi :S\rightarrow\mathbb{C}$ satisfying 
$$\phi(xy)=\phi(x)\chi(y)+\phi(y)\chi(x),$$
such that 
$$\left\{\begin{array}{l}
f=\phi +\chi \\
g=\chi
\end{array}.\right.$$
\item[(8)] There exist a non-zero multiplicative function $\chi :S\rightarrow\mathbb{C}$ such that $\chi\neq \chi^*$, $f=\dfrac{\chi+\chi^*}{2}$, and $g\in \left\lbrace \chi , \chi^*\right\rbrace$.\\
Note that $f$ and $g$ are Abelian in each case.\par
Furthermore, if $S$ is a topological semigroup and $f,g\in C(S)$, then \\$\chi,\phi,\chi_1,\chi_2,\chi^*\in C(S)$.
\end{enumerate}
\label{Thm}
\end{thm}
\begin{proof}
If $f=0$, it is easy to see that $g=0$. This is case (1), so from now on we assume that $f\neq 0$. Suppose that $f=0$ on $S^2$, then according to \cite[Lemma 3.1]{ST2} we get that $g=0$ or $g=2f$. This occurs in parts (2) and (3). In the rest of the proof we assume that $f\neq 0$ on $S^2$ and we  discuss two cases according to whether $f$ and  $g$ are linearly dependent or not.\\
\underline{First case :} $f$ and $g$ are linearly dependent. There exists a constant $\delta  \in \mathbb{C}$ such that $g=\delta f$. Equation \eqref{E2} can be written as 
\begin{equation}
f(x\sigma (y))=\left(2\delta -\delta^2 \right) f(x)f(y).
\label{T1}
\end{equation}
Since $f\neq 0$ on $S^2$, then $\delta  \in \mathbb{C}\backslash \lbrace 0, 2\rbrace$. According to Lemma \ref{M}, the function $\chi :=\left(2\delta -\delta^2 \right) f$ is multiplicative and $\chi^*=\chi$, so $f=\dfrac{1}{2\delta -\delta^2}\chi$ and $g=\dfrac{\delta}{2\delta-\delta^2}\chi$. If we put $\beta=\dfrac{\delta}{2\delta-\delta^2}$, then \\
\[f=\dfrac{\beta^2}{2\beta-1}\quad\text{and}\quad g=\beta\chi,\]
such that $\beta \in \mathbb{C}\backslash \left\lbrace 0,\dfrac{1}{2}\right\rbrace $. This is case (4).\\
\underline{Second case :} $f$ and $g$ are linearly independent. According to Lemma \ref{Lem1} (3)(b) and (c), there exists two constants $\alpha ,\beta \in \mathbb{C}$ such that 
\begin{equation}
g=(1-\beta)g^*+\beta f^*\quad \text{and}\quad \alpha (f^*-g^*)=f-g,
\label{S1}
\end{equation}
 and $\alpha \in \lbrace -1, 1\rbrace$.\\
\underline{Subcase A:} $\alpha =1$. According to Lemma \ref{Lem1} (3)(c) $\beta=0$, $f^*=f$ and $g^*=g$.\\
 So equation \eqref{E2} can be written as $$f(xy)=f(x)g(y)+f(y)g(x)-g(x)g(y).$$
Then according to \cite[Proposition 3.1]{Ebanks} and taking into account that  $f$ and $g$ are linearly independent, we have the following possibilities:\\
(i) $f=\dfrac{\chi_1+\chi_2}{2}+\dfrac{c_1^2+1}{2c_1}\dfrac{\chi_1-\chi_2}{2}$ and $g=\dfrac{\chi_1+\chi_2}{2}+c_1\dfrac{\chi_1-\chi_2}{2}$, where $\chi_1, \chi_2:S\rightarrow \mathbb{C}$ are two different non-zero multiplicative functions and $c_1 \in \mathbb{C}\backslash \lbrace 0\rbrace$ is a constant. Since $f\neq g$, then $c_1 \notin \lbrace -1,1\rbrace$. $f=f^*$ and $g=g^*$ implies that $\chi_1=\chi_1^*$ and $\chi_2=\chi_2^*$. This is case (5).\\
(ii) There exist a non-zero multiplicative function $\chi :S\rightarrow\mathbb{C}$ and a non-zero function $\phi :S\rightarrow\mathbb{C}$ satisfying 
$$\phi(xy)=\phi(x)\chi(y)+\phi(y)\chi(x),$$
such that 
$$\left\{\begin{array}{l}
f=\dfrac{1}{2}\phi +\chi, \\
g=\phi +\chi.\end{array}\right.\quad \text{or}\quad\left\{\begin{array}{l}
f=\phi +\chi, \\
g=\chi.
\end{array}\right.$$
Since $f=f^*$ and $g=g^*$, then $\chi=\chi^*$ and $\phi=\phi^*$. This occurs in parts (6) and (7).\\

\underline{Subcase B:} $\alpha =-1$. In this case, we get $f^*-g^*=g-f$, so $f^e=g^e$. On the other hand equation \eqref{E2} implies that
\begin{equation}
f(xy)=f(x)g^*(y)+g(x)\left( g(y)-f(y)\right),
\label{S2} 
\end{equation}
\begin{equation}
f^*(xy)=f^*(x)g(y)+g^*(x)\left(f(y)-g(y) \right).
\label{S3} 
\end{equation}
According to Lemma \ref{Lem2}(2), $f(xyz)=f^*(xyz)$ for all $x,y,z\in S$. Replacing $y$ by $yz$ in \eqref{S2} and \eqref{S3}, we obtain
\begin{equation}
\left(g(yz)-f(yz) \right)\left(g(x)+g^*(x) \right) =f^*(x)g(yz)-f(x)g^*(yz).
\label{qor1}
\end{equation}

In the identity \eqref{qor1}, the left hand side is an even function of $x$, so
\[f^*(x)g(yz)-f(x)g^*(yz)=f(x)g(yz)-f^*(x)g^*(yz),\]
this implies that $f(x)g^e(yz)=f^*(x)g^e(yz)$. If $g^e=0$ on $S^2$, then $f^e=0$ on $S^2$, that is $f(xy)=-f^*(xy)$ for all $x,y\in S$, so by using \eqref{S2} and \eqref{S3} we get
\begin{equation}
f^*(x)g(y)+f(x)g^*(y)=2\left[f(y)-g(y) \right]g^{\circ}(x).
\label{qor2}
\end{equation}

In the identity \eqref{qor2}, the right hand side is an odd function of $x$, so
\[f^*(x)g(y)+f(x)g^*(y)=-f(x)g(y)-f^*(x)g^*(y).\]
This implies that $f^*(x)g^e(y)=-f(x)g^e(y)$, so that $2f^e(x)g^e(y)=0$. Since $f^e=g^e$, we deduce that $f^e=0$. So $f^*=-f$ and $g^*=-g$, and so we can deduce from \eqref{S1} that $(2-\beta)g+\beta f=0$. This is a contradiction, since $f$ and $g$ are linearly independent. So $g^e\neq 0$ on $S^2$ and $f=f^*$. Now \eqref{S3} becomes 
\begin{equation}
f(xy)=f(x)g(y)+g^*(x)\left(f(y)-g(y) \right).
\label{S4}
\end{equation}
If we replace $y$ by $\sigma(y)$ in \eqref{S4}, we get
\begin{equation}
f(x\sigma(y))=f(x)g^*(y)+g^*(x)\left(g(y)-f(y) \right).
\label{S5}
\end{equation}
By adding \eqref{S4} to \eqref{S5}, we get
\begin{equation}
f(xy)+f(x\sigma(y))=2f(x)g^e(y).
\label{S6}
\end{equation}
Since $f=f^*$, then according to Lemma \ref{Lem2}(1), f is central. On the other hand $g^e=f^e=f$, so equation \eqref{S6} can be written as 
\begin{equation}
f(xy)+f(\sigma(y)x)=2f(x)f(y).
\end{equation}
According to \cite[Theorem 2.1]{ST}, there exists a non-zero multiplicative function $\chi:S\rightarrow\mathbb{C}$ such that \begin{equation}
f=\dfrac{\chi+\chi^*}{2}.
\label{F}
\end{equation}
Since $f=g^e$, $g^*=g^e-g^{\circ}$ and $g=g^e+g^{\circ}$ equation \eqref{S2} can be written as 
\[g^e(xy)=g^e(x)\left(g^e(y)-g^{\circ}(y) \right)+g^{\circ}(y)\left(g^e(x)+g^{\circ}(x) \right).\]
This implies that 
\begin{equation}
g^e(xy)=g^e(x)g^e(y)+g^{\circ}(x)g^{\circ}(y).
\label{S7}
\end{equation}
 By substituting $g^e=f$ from \eqref{F} into \eqref{S7} we find that
\[g^{\circ}(x)g^{\circ}(y)=\dfrac{\chi(x)-\chi^*(x)}{2}\dfrac{\chi(y)-\chi^*(y)}{2},\]
which implies that $g^{\circ}=c(\chi-\chi^*)$ for some constant $c\in \mathbb{C}$. Then 
\begin{equation}
g=g^e+g^{\circ}=\left(c +\dfrac{1}{2} \right)\chi+ \left(\dfrac{1}{2}-c \right)\chi^*.
\label{G}
\end{equation}
By substituting \eqref{F} and \eqref{G} into equation \eqref{E2} we get after some simplification that
\[\left(\dfrac{1}{4}-c^2 \right)\left[\chi(x)-\chi^*(x) \right]\left[\chi(y)-\chi^*(y) \right]=0.\]
Since $\chi\neq \chi^*$, we deduce that $\dfrac{1}{4}-c^2=0$. That is $c \in \left\lbrace  \dfrac{1}{2}, \dfrac{-1}{2}\right\rbrace $. So\\
$g\in \left\lbrace \chi , \chi^*\right\rbrace$.
This is case (8).\par
 Conversely we check by elementary computations that if $f, g$ have one of the forms (1)--(8) then $(f,g)$ is a solution of equation \eqref{E2}.\par 
 For the topological statements the cases (1), (2), (3) and (4) are trivial. In the cases (5), (6) and (7) we get the continuity of $\chi_1,\chi_2$ and $\chi$  by proceeding exactly as in the proof of \cite[Corollary 3.2]{Ebanks}, and then $\phi$ is continuous since $\phi =f-\chi$ or $\phi=2(f-\chi)$ and $f$ is continuous. For the case (8) we get the continuity of $\chi$ and $\chi^*$ by the help of \cite[Theorem 3.18]{ST2}.  This completes the proof of Theorem \ref{Thm}.
\end{proof}
The following corollary follows directly from Theorem \ref{Thm} and \cite[Theorem 3.1]{EB2}.
\begin{cor}
Let $S$ be a semigroup. The solutions $f,g :S\rightarrow\mathbb{C}$ of \eqref{E2} are the following pairs :
\begin{enumerate}
\item[(1)] $f=0$ and $g=0$.
\item[(2)] $f$ is any non-zero function such that $f=0$ on $S^2$, while $g=0$.
\item[(3)]  $f$ is any non-zero function such that $f=0$ on $S^2$, while $g=2f$.
\item[(4)] $f=\dfrac{\beta^2}{2\beta -1}\chi$ and $g=\beta\chi$, where $\chi$ is a non-zero even multiplicative function and $\beta  \in \mathbb{C}\backslash \left\lbrace  0, \dfrac{1}{2}\right\rbrace $ is a constant.
\item[(5)] There exist two different, non-zero even multiplicative functions $\chi_1, \chi_2:S\rightarrow\mathbb{C}$ and a constant $c_1\in \mathbb{C}\backslash\lbrace0,-1,1\rbrace$ such that
$$f=\dfrac{\chi_1+\chi_2}{2}+\dfrac{c_1^2+1}{2c_1}\dfrac{\chi_1-\chi_2}{2}\quad\text{and}\quad g=\dfrac{\chi_1+\chi_2}{2}+c_1\dfrac{\chi_1-\chi_2}{2}.$$
\item[(6)] 
$$\left\{\begin{array}{l}
f=\dfrac{1}{2}\chi A +\chi \quad\text{on}\quad S\backslash I_{\chi}\\
f=0\quad\quad\quad\quad\ \text{on}\quad I_{\chi}\backslash P_{\chi}\\
f=\dfrac{1}{2}\rho \quad\quad\quad\ \text{on}\quad P_{\chi}\end{array}\right.\quad \text{and}\quad \left\{\begin{array}{l}
 g=\chi +\chi A \quad\ \text{on}\quad S\backslash I_{\chi}\\
 g=0 \quad\quad\quad\quad\text{on}\quad I_{\chi}\backslash P_{\chi}\\
g=\rho \quad\quad\quad\quad\text{on}\quad P_{\chi}\end{array},\right.$$
where  $\chi: S \rightarrow \mathbb{C}$ is a non-zero multiplicative function and $A: S \backslash I_{\chi} \rightarrow \mathbb{C}$ is a non-zero additive function such that $\chi^*=\chi, A\circ \sigma =A$, $\rho: P_{\chi} \rightarrow \mathbb{C}$ is an even function with conditions (I) and (II) holding.
\item[(7)]
$$\left\{\begin{array}{l}
f=\chi A +\chi \quad\text{on}\quad S\backslash I_{\chi}\\
f=0\quad\quad\quad\ \ \text{on}\quad I_{\chi}\backslash P_{\chi}\\
f=\rho\quad\quad\quad\ \ \text{on}\quad P_{\chi}\end{array}\right.\quad\text{and}\quad g=\chi,$$
where  $\chi: S \rightarrow \mathbb{C}$ is a non-zero multiplicative function and $A: S \backslash I_{\chi} \rightarrow \mathbb{C}$ is a non-zero additive function such that $\chi^*=\chi, A\circ \sigma =A$, $\rho: P_{\chi} \rightarrow \mathbb{C}$ is an even function with conditions (I) and (II) holding.
\item[(8)] $f=\dfrac{\chi+\chi^*}{2}$, and $g\in \left\lbrace \chi , \chi^*\right\rbrace$, where  $\chi: S \rightarrow \mathbb{C}$ is a non-zero multiplicative function  such that $\chi^*\neq\chi$.\\
Note that $f$ and $g$ are Abelian in each case.\par
Furthermore, if $S$ is a topological semigroup and $f,g\in C(S)$, then $\chi,\chi_1,\chi_2,\chi^*\in C(S)$, $A\in C(S\backslash I_{\chi})$ and $\rho \in C(P_{\chi})$.
\end{enumerate}
\label{Cor1}
\end{cor}
\section{Solutions of Equations \eqref{App1}}
In this section, we solve the functional equation \eqref{App1} on semigroups.
\begin{thm}
The solutions $f,g:S\rightarrow\mathbb{C}$ of equation \eqref{App1} can be listed as follows :
\begin{enumerate}
\item[(1)] $f=0$ and $g=0$.
\item[(2)] $f$ is any non-zero function such that $f=0$ on $S^2$, and $g=0$.
\item[(3)]$g$ is any non-zero function such that $g=0$ on $S^2$, and $f=0$.
\item[(4)] $f=\dfrac{\alpha\beta}{1-2\beta}\chi$ and $g=\beta\chi$, where $\chi$ is a non-zero even multiplicative function and $\beta  \in \mathbb{C}\backslash \left\lbrace  0, \dfrac{1}{2}\right\rbrace $ is a constant.
\item[(5)] $$f=\dfrac{-\alpha}{2}\left(\chi_1+\chi_2 \right) +\dfrac{-\alpha (c_1^2+2)}{4c_1}\left( \chi_1-\chi_2\right)\ \text{and}\  g=\dfrac{\chi_1+\chi_2}{2}+c_1\dfrac{\chi_1-\chi_2}{2},$$
where $\chi_1, \chi_2:S\rightarrow\mathbb{C}$ are two different, non-zero even multiplicative functions  and $c_1\in \mathbb{C}\backslash\lbrace0,-1,1\rbrace$ is a constant.
\item[(6)] $$\left\{\begin{array}{l}
g=\chi +\chi A \ \quad\text{on}\quad S\backslash I_{\chi}\\
g=0 \quad\quad\quad\quad\text{on}\quad I_{\chi}\backslash P_{\chi}\\
g=\rho \quad\quad\quad\quad\text{on}\quad P_{\chi}\end{array}\right.\quad\text{and}\quad f=\alpha \chi,$$
where  $\chi: S \rightarrow \mathbb{C}$ is a non-zero multiplicative function and $A: S \backslash I_{\chi} \rightarrow \mathbb{C}$ is a non-zero additive function such that $\chi^*=\chi, A\circ \sigma =A$, $\rho: P_{\chi} \rightarrow \mathbb{C}$ is an even function with conditions (I) and (II) holding.
\item[(7)] $$\left\{\begin{array}{l}
f=-2\alpha\chi A -\alpha\chi \quad\text{on}\quad S\backslash I_{\chi}\\
f=0\quad\quad\quad\quad\quad\quad\text{on}\quad I_{\chi}\backslash P_{\chi}\\
f=-2\alpha\rho \quad\quad\quad\quad\text{on}\quad P_{\chi}\end{array}\right.\quad\text{and}\quad g=\chi,$$
where  $\chi: S \rightarrow \mathbb{C}$ is a non-zero multiplicative function and $A: S \backslash I_{\chi} \rightarrow \mathbb{C}$ is a non-zero additive function such that $\chi^*=\chi, A\circ \sigma =A$, $\rho: P_{\chi} \rightarrow \mathbb{C}$ is an even function with conditions (I) and (II) holding.
\item[(8)] $$\left\{\begin{array}{l}
f=-\alpha\chi \\
g=\chi^*\end{array}\right.\quad \text{or}\quad \left\{\begin{array}{l}
f=-\alpha\chi^* \\
g=\chi\end{array},\right.$$
where  $\chi: S \rightarrow \mathbb{C}$ is a non-zero multiplicative function  such that $\chi^*\neq\chi$.\\
Note that $f$ and $g$ are Abelian in each case.\par
Furthermore, if $S$ is a topological semigroup and $f,g\in C(S)$, then \\$\chi,\chi_1,\chi_2,\chi^*\in C(S)$, $A\in C(S\backslash I_{\chi})$ and $\rho \in C(P_{\chi})$.
\end{enumerate}
\label{Appli1}
\end{thm}
\begin{proof}
Let $(f,g)$ be a solution of equation \eqref{App1}, we define the function \\$F:=\dfrac{1}{\alpha}f-g$. For all $x,y\in S$, we have
\begin{equation}
F(x\sigma(y))=\dfrac{1}{\alpha}f(x\sigma(y))-g(x\sigma(y))=\dfrac{1}{\alpha}f(x)g(y)+\dfrac{1}{\alpha}f(y)g(x).
\label{Ap1}
\end{equation}
On the other hand, we have 
\begin{equation}
F(x)g(y)+F(y)g(x)=\dfrac{1}{\alpha}f(x)g(y)+\dfrac{1}{\alpha}f(y)g(x)-2g(x)g(y).
\label{Ap2}
\end{equation} 
By taking into account equation \eqref{Ap2}, equation \eqref{Ap1} becomes
\[F(x\sigma(y))=F(x)g(y)+F(y)g(x)+2g(x)g(y).\]
This implies that 
\begin{equation}
\dfrac{-1}{2}F(x\sigma(y))=\dfrac{-1}{2}F(x)g(y)+\dfrac{-1}{2}F(y)g(x)-g(x)(y).
\label{Ap3}
\end{equation}
This means that $\left( \dfrac{-1}{2}F,g\right) $ satisfies equation \eqref{E2}. According to Corollary \ref{Cor1}, we have the following possibilities :\\
(1) $\dfrac{-1}{2}F=0$ and $g=0$. Since $F=\dfrac{1}{\alpha}f-g$ we deduce that $f=0$. This is case (1).\\
(2) $\dfrac{-1}{2}F$ is any non-zero function such that $\dfrac{-1}{2}F=0$ on $S^2$, and $g=0$. 
This implies that $f$ is any non-zero function such that $f=0$ on $S^2$. This occurs in part (2).\\
(3) $\dfrac{-1}{2}F$ is any non-zero function such that $\dfrac{-1}{2}F=0$ on $S^2$, and $g=-F$. 
That is $f=g=0$ on $S^2$, and $\dfrac{-1}{\alpha}f=0$ on $S \backslash S^2$. So we deduce that $f=0$ and $g$ is any non-zero function such that $g=0$ on $S^2$. This is case (3).\\
(4) $\dfrac{-1}{2}F=\dfrac{\beta^2}{2\beta -1}\chi$ and $g=\beta\chi$, where $\chi$ is a non-zero multiplicative function such that $\chi^*=\chi$ and $\beta  \in \mathbb{C}\backslash \left\lbrace  0, \dfrac{1}{2}\right\rbrace $ is a constant. This implies that 
\[\dfrac{-1}{2\alpha}f=\dfrac{\beta^2}{2\beta -1}\chi-\dfrac{\beta}{2}\chi.\]
Then we get $f=\dfrac{\alpha\beta}{1-2\beta}\chi$. This occurs in case (4).\\
(5) $$\dfrac{-1}{2}F=\dfrac{\chi_1+\chi_2}{2}+\dfrac{c_1^2+1}{2c_1}\dfrac{\chi_1-\chi_2}{2}\quad\text{and}\quad g=\dfrac{\chi_1+\chi_2}{2}+c_1\dfrac{\chi_1-\chi_2}{2},$$
where $\chi_1, \chi_2:S\rightarrow\mathbb{C}$ are two different, non-zero even multiplicative functions  and $c_1\in \mathbb{C}\backslash\lbrace0,-1,1\rbrace$ is a constant. So we deduce that 
$$\dfrac{-1}{2\alpha}f=\dfrac{\chi_1+\chi_2}{4}+\dfrac{c_1^2+2}{4c_1}\dfrac{\chi_1-\chi_2}{2}.$$
That is $f=\dfrac{-\alpha}{2}\left(\chi_1+\chi_2 \right) +\dfrac{-\alpha (c_1^2+2)}{4c_1}\left( \chi_1-\chi_2\right) $. This is case (5).\\
(6) $$\dfrac{-1}{2}F=\left\{\begin{array}{l}
\chi A +\chi\quad\ \text{on}\quad S\backslash I_{\chi}\\
0\quad\quad\quad\quad\text{on}\quad I_{\chi}\backslash P_{\chi}\\
\dfrac{1}{2}\rho\ \quad\quad\quad\text{on}\quad P_{\chi}\end{array}\right.\quad\text{and}\quad g=\left\{\begin{array}{l}
\chi +\chi A \quad\text{on}\quad S\backslash I_{\chi}\\
0 \quad\quad\quad\quad\text{on}\quad I_{\chi}\backslash P_{\chi}\\
 \rho \quad\quad\quad\quad\text{on}\quad P_{\chi}\end{array},\right.$$
where  $\chi: S \rightarrow \mathbb{C}$ is a non-zero multiplicative function and $A: S \backslash I_{\chi} \rightarrow \mathbb{C}$ is a non-zero additive function such that $\chi^*=\chi, A\circ \sigma =A$, $\rho: P_{\chi} \rightarrow \mathbb{C}$ is an even function with conditions (I) and (II) holding. Then we get 
$$\dfrac{-1}{2\alpha}f=\left\{\begin{array}{l}
\dfrac{-1}{2}\chi \ \quad\text{on}\quad S\backslash I_{\chi}\\
0 \quad\quad\quad\text{on}\quad I_{\chi}\backslash P_{\chi}\\
0\quad\quad\quad\text{on}\quad P_{\chi}\end{array}.\right.$$
This implies that $f=\alpha \chi$. This occurs in case (6).\\
(7) $$\dfrac{-1}{2}F=\left\{\begin{array}{l}
\chi A +\chi\quad\ \text{on}\quad S\backslash I_{\chi}\\
0\quad\quad\quad\quad\text{on}\quad I_{\chi}\backslash P_{\chi}\\
\rho\quad\quad\quad\quad\text{on}\quad P_{\chi}\end{array}\right.\quad\text{and}\quad g=\chi,$$
where  $\chi: S \rightarrow \mathbb{C}$ is a non-zero multiplicative function and $A: S \backslash I_{\chi} \rightarrow \mathbb{C}$ is a non-zero additive function such that $\chi^*=\chi, A\circ \sigma =A$, $\rho: P_{\chi} \rightarrow \mathbb{C}$ is an even function with conditions (I) and (II) holding. That is 
$$\dfrac{-1}{2\alpha}f=\left\{\begin{array}{l}
\chi A +\dfrac{1}{2}\chi\quad\quad\text{on}\quad S\backslash I_{\chi}\\
0\quad\quad\quad\quad\quad\ \text{on}\quad I_{\chi}\backslash P_{\chi}\\
\rho\quad\quad\quad\quad\quad\ \text{on}\quad P_{\chi}\end{array}.\right.$$
So we deduce that 
$$f=\left\{\begin{array}{l}
-2\alpha\chi A -\alpha\chi\quad\text{on}\quad S\backslash I_{\chi}\\
0\quad\quad\quad\quad\quad\quad\ \text{on}\quad I_{\chi}\backslash P_{\chi}\\
-2\alpha\rho\quad\quad\quad\quad\ \text{on}\quad P_{\chi}\end{array}.\right.$$
This is part (7).\\
(8) $\dfrac{-1}{2}F=\dfrac{\chi+\chi^*}{2}$, and $g\in \left\lbrace \chi , \chi^*\right\rbrace$, where  $\chi: S \rightarrow \mathbb{C}$ is a non-zero multiplicative function  such that $\chi^*\neq\chi$. This implies that\\ 
$\dfrac{-1}{2\alpha}f\in \left\lbrace \dfrac{1}{2}\chi , \dfrac{1}{2}\chi^*\right\rbrace$. That is 
\[f\in \left\lbrace -\alpha\chi , -\alpha\chi^*\right\rbrace.\]
So we deduce that  
$$\left\{\begin{array}{l}
f=-\alpha\chi \\
g=\chi^*\end{array}\right.\quad \text{or}\quad \left\{\begin{array}{l}
f=-\alpha\chi^* \\
g=\chi\end{array}.\right.$$
This occurs in part (8).\par
 The converse and the continuity statements are easy to verify. This completes the proof of Theorem \ref{Appli1}.
\end{proof}
\section{Solutions of Equation \eqref{App2}}
In the following theorem we extend the results obtained in \cite[Proposition 4.1]{Z} on topological groups to the case of semigroups.
\begin{thm}
The solutions $f,g :S\rightarrow\mathbb{C}$ of the functional equation \eqref{App2} are the following pairs :
\begin{enumerate}
\item[(1)] $f=\alpha g$ and $g$ is arbitrary.
\item[(2)] $g$ is any non-zero function such that $g=0$ on $S^2$, and $f=0$.
\item[(3)] $f$ is any non-zero function such that $f=0$ on $S^2$, and $g=0$.
\item[(4)] $f$ is any non-zero function such that $f=0$ on $S^2$, and $g=\dfrac{c}{\alpha (1+c)}f$, 
where $c \in \mathbb{C}\backslash \lbrace 0,-1\rbrace$ is a constant.
\item[(5)] $$f=\alpha \dfrac{\chi +\chi ^*}{2}+\alpha(c_1+c_2)\dfrac{\chi -\chi ^*}{2}\ \ \text{and}\  \ g=\dfrac{\chi +\chi ^*}{2}+c_2\dfrac{\chi -\chi ^*}{2},$$
where $\chi  : S\rightarrow \mathbb{C}$ is a multiplicative function such that $\chi ^*\neq \chi$ and $c_1 \in \mathbb{C}\backslash  \lbrace0\rbrace,c_2\in \mathbb{C}$ are constants.
\item[(6)] $$
f=\left\{\begin{array}{l}
\alpha \chi(1+(1+c)A)\ \text { on } \  S \backslash I_{\chi} \\
0\ \quad\quad\quad \quad\quad\quad\quad\text { on } \  I_\chi \backslash P_{\chi} \\
\alpha (1+c)\rho\  \quad\quad\quad\ \text { on } \  P_\chi 
\end{array}\right.\quad\text{and}\quad g=\left\{\begin{array}{l}
\chi(1+c A) \ \ \text { on }\  S \backslash I_{\chi} \\
0 \  \quad\quad\quad\quad\ \text { on } \  I_\chi \backslash P_{\chi} \\
c\rho \  \quad\quad\quad\quad\text { on } \  P_\chi
\end{array},\right.
$$
where $c \in \mathbb{C}$ is a constant, $\chi: S \rightarrow \mathbb{C}$ is a non-zero multiplicative function and $A: S \backslash I_{\chi} \rightarrow \mathbb{C}$ is a non-zero additive function such that $\chi^{*}=\chi$, $A \circ \sigma=-A$, and $\rho: P_{\chi} \rightarrow \mathbb{C}$ is an odd function with conditions (I) and (II) holding.\\
Note that, off the exceptional case (1) $f$ and $g$ are Abelian.\par
Furthermore, if $S$ is a topological semigroup and $f,g\in C(S)$, then \\$\chi,\chi^*\in C(S)$, $A\in C(S\backslash I_{\chi})$ and $\rho \in C(P_{\chi})$.
\end{enumerate}
\label{Appli2}
\end{thm}
\begin{proof}
Let $f,g :S\rightarrow\mathbb{C}$ be a solution of the functional equation \eqref{App2}. If we put $F:=\dfrac{1}{\alpha}f-g$, then for all $x,y\in S$ we have 
\begin{equation}
F(x\sigma(y))=\dfrac{1}{\alpha}f(x\sigma(y))-g(x\sigma(y))=\dfrac{1}{\alpha}f(x)g(y)-\dfrac{1}{\alpha}f(y)g(x).
\label{ap1}
\end{equation}
On the other hand 
\begin{equation}
F(x)g(y)-F(y)g(x)=\dfrac{1}{\alpha}f(x)g(y)-g(x)g(y)-\dfrac{1}{\alpha}f(y)g(x)+g(x)g(y).
\label{ap2}
\end{equation}
Taking into account equation \eqref{ap2}, equation \eqref{ap1} becomes 
\[F(x\sigma(y))=F(x)g(y)-F(y)g(x),\quad \text{for all}\quad x,y\in S.\]
If $F=0$, the $g$ will be arbitrary and $f=\alpha g$. This is case (1). From now on we assume that $F\neq 0$. According to \cite[Proposition 3.2]{Ase}, the pairs $(F,g)$ falls into the categories :\\
(1) $F$ is any non-zero function such that $F=0$ on $S^2$, and $g=cF$, where $c \in \mathbb{C}$ is a constant. This implies that $f=g=0$ on $S^2$ and $(1+c)g=\dfrac{c}{\alpha}f$ on $S \backslash S^2$.
If $c=-1$, then we get that $f=0$ and $g$ is any non-zero function such that $g=0$ on $S^2$. This occurs in case (2). Now if $c=0$, then we get $g=0$ and $f$ is any non-zero function such that $f=0$ on $S^2$. This is case (3). Now if $c\notin \lbrace 0,-1\rbrace$, then $g=\dfrac{c}{\alpha (1+c)}f$ where $f$ is any non-zero function such that $f=0$ on $S^2$. This is part (4).\\
(2)  $$F=c_1\dfrac{\chi -\chi ^*}{2}\quad\text{and}\quad g=\dfrac{\chi +\chi ^*}{2}+c_2\dfrac{\chi -\chi ^*}{2},$$
where $\chi  : S\rightarrow \mathbb{C}$ is a multiplicative function such that $\chi ^*\neq \chi$ and $c_1 \in \mathbb{C}\backslash  \lbrace0\rbrace,c_2\in \mathbb{C}$ are constants. Since $F=\dfrac{1}{\alpha}f-g$, then we deduce that 
\[f=\alpha \dfrac{\chi +\chi ^*}{2}+\alpha(c_1+c_2)\dfrac{\chi -\chi ^*}{2}.\]
This is case (5).\\
(3) $$
F=\left\{\begin{array}{l}
\chi A\quad \text { on } \quad S \backslash I_{\chi} \\
0\quad \quad\text { on } \quad I_\chi \backslash P_{\chi} \\
\rho\quad\quad \text { on } \quad P_\chi 
\end{array}\right.\quad\text{and}\quad g=\left\{\begin{array}{l}
\chi(1+c A) \quad\  \text { on } \quad S \backslash I_{\chi} \\
0 \quad \quad\quad\quad\quad\text { on } \quad I_\chi \backslash P_{\chi} \\
c\rho  \quad\quad\quad\quad\ \  \text { on } \quad P_\chi  
\end{array},\right.
$$
where $c \in \mathbb{C}$ is a constant, $\chi: S \rightarrow \mathbb{C}$ is a non-zero multiplicative function and $A: S \backslash I_{\chi} \rightarrow \mathbb{C}$ is a non-zero additive function such that $\chi^{*}=\chi$, $A \circ \sigma=-A$, and $\rho: P_{\chi} \rightarrow \mathbb{C}$ is an odd function with conditions (I) and (II) holding. This implies that 
$$
f=\left\{\begin{array}{l}
\alpha \chi(1+(1+c)A)\quad \text { on } \quad S \backslash I_{\chi} \\
0\quad\quad\quad\quad\quad\quad\quad\quad \text { on } \quad I_\chi \backslash P_{\chi} \\
\alpha (1+c)\rho\quad\quad\quad\quad \ \text { on } \quad P_\chi
\end{array}.\right.
$$
This occurs in case (6).\par 
 Conversely we check by elementary computations that if $(f,g)$ have one of the forms (1)--(6), then $(f,g)$ is a solution of equation \eqref{App2}.\par 
 The topological statements are easily verified. This completes the proof of Theorem \ref{Appli2}.
\end{proof}
\section{Examples}
In this section we give two examples of solutions of \eqref{E3}, \eqref{E1}, \eqref{E2}, \eqref{App1} and \eqref{App2}. The first one is for a semigroup $S$ such that $S^2\neq S$ and the second one is where the semigroup $S$ satisfies $S^2=S$.
\begin{ex}
Let $S=(\mathbb{N}\backslash \lbrace 1\rbrace , .)$ and $p,q$ be two distinct primes, and define $\sigma :S \rightarrow S$ by $\sigma (x)=\widehat{x}$, where $\widehat{x}$ is obtained from $x$ by replacing every copy of $p$ in the prime factorization of $x$ by $q$, and vice versa. Let $I=p\mathbb{N}\cup q\mathbb{N}$ and define $\chi :S\rightarrow \mathbb{C}$ by 
$$\chi(x):=\left\{\begin{array}{l}
1\quad \text { for } \quad x\in S \backslash I \\
0\quad  \text { for } \quad x\in I
\end{array}.\right.$$
Then $\chi$ is multiplicative, $I_{\chi}=I$ and $\chi^*=\chi$. $S$ is not generated by its squares and $S^2=S\backslash \mathbb{P}$, where $\mathbb{P}$ denote the set of all prime numbers.  $P_{\chi}=I\backslash I^2$. Now let $c\in \mathbb{C}\backslash \lbrace 0\rbrace$ and define $\rho : P_{\chi}\rightarrow \mathbb{C}$ by 
$$\rho (x):=\left\{\begin{array}{l}
c\quad \text { for } \quad x=pw\quad\text{with}\quad w\in S \backslash I  \\
c\quad  \text { for } \quad x=qw\quad\text{with}\quad w\in S \backslash I 
\end{array},\right.$$
then $\rho\circ\sigma=\rho$. For each $p\in \mathbb{P}$ define $C_p:\mathbb{N}\rightarrow \mathbb{N}\backslash \lbrace 0\rbrace$ by \\
$C_p(x):=$ the number of copies of $p$ occurring in the prime factorisation of $x$. For each $x\in S$ let $P_x$ denote the set of prime factors of $x$. The additive functions on $S \backslash I$ have the form $A(x)=\sum\nolimits_{p\in {{P}_{x}}\backslash I}{A(p){{C}_{p}}(x)}$. So $A\circ \sigma=A$.\par 
 We get the solutions $f,g:S\rightarrow\mathbb{C}$ of equation \eqref{E3}, \eqref{E1}, \eqref{E2}, \eqref{App1} and equation \eqref{App2} by plugging the appropriate forms above into the formulas of Theorem \ref{TE3}, Theorem \ref{P1}, Theorem \ref{Thm}, Theorem \ref{Appli1} and Theorem \ref{Appli2} respectively.
\end{ex}
\begin{ex}
Let $S=\left]-1,1 \right[ \times\left] -1,1\right[ $ under coordinatwise multiplication and the usual topology, and let $\sigma : S\rightarrow S$ be the switching involution $\sigma (x,y)=\sigma(y,x)$. $S$ is not generated by its squares but $S^2=S$.\\
If $\chi$ is a non-zero continuous multiplicative function on $]-1,1[$, then $\chi$ have one of the forms
$$
{{\chi }_{0}}:=1,\ {{\chi }_{\alpha }}(x):=\left\{ \begin{matrix}
   {{\lvert x \rvert}^{\alpha }} & for\ x\ne 0  \\
   0 & for\ x=0  \\
\end{matrix}\ \ or\   \right.\ {{\chi }_{\alpha }}'(x):=\left\{ \begin{matrix}
   {{\lvert x \rvert}^{\alpha }}sgn (x) & for\ x\ne 0  \\
   0 & for\ x=0  \\
\end{matrix}\ , \right.
$$
for some $\alpha \in \mathbb{C}$ such that $\mathcal{R}(\alpha )>0$, where $\mathcal{R}(\alpha )$ denote the real part of $\alpha$.\\ The continuous multiplicative functions on $S$ have the form $\chi=\chi_1\otimes \chi_2$ where $\chi_1, \chi_2 \in C\left(]-1,1[ \right) $. The prime ideals on $S$ serving as null ideals of continuous multiplicative functions are $I_1=\lbrace 0\rbrace\times ]-1,1[$, $I_2=]-1,1[\times \lbrace 0\rbrace$, and $I_1\cup I_2$.\\
The continuous even multiplicative functions on $S$ are those of the form $\chi_0\otimes \chi_0$ , $\chi_{\alpha}\otimes \chi_{\alpha}$, or $\chi_{\alpha}'\otimes \chi_{\alpha}'$ with $\mathcal{R}(\alpha )>0$. Let $\chi_1=\chi_0\otimes \chi_0$ , $\chi_2=\chi_{\alpha}\otimes \chi_{\alpha}$, and $\chi_3=\chi_{\alpha}'\otimes \chi_{\alpha}'$, if $I_{\chi}$ denote the null ideal of a multiplicative function $\chi$, then $I_{\chi_2}=I_{\chi_3}=I_1\cup I_2$ and $I_{\chi_1}=\emptyset$.\\
The continuous additive functions on $S\backslash (I_1\cup I_2)$ have the form $A(x,y)=a\log\lvert x\rvert+b\log \lvert y\rvert$ for $xy\neq 0$. Such functions are even and non-zero if and only if $b=a\neq 0$.\par
The solutions $f, g\in C(S)$ such that $g\neq 0$ and $f\neq 0$ of equation \eqref{E3}, \eqref{E1}, \eqref{E2}, \eqref{App1} and equation \eqref{App2} are obtained by plugging the forms above into the formulas of Theorem \ref{TE3}, Theorem \ref{P1}, Theorem \ref{Thm}, Theorem \ref{Appli1} and Theorem \ref{Appli2} respectively.
\end{ex}
\subsection*{Acknowledgment}

\textbf{Statements and Declarations}\\
\\
\textbf{Author contributions} This work is done by the authors solely.\\
\\
\textbf{Availability and requirements} Not Applicable.\\
\\
\textbf{Competing Interests} The authors declare that they have no competing interest.\\
\\
\textbf{Funding} This research has no external funding.\\
\\
\textbf{Conflicts of Interest} The authors declare no conflict of interest.


\begin{thebibliography}{1}
\bibitem{Acz} Aczél, J., Dhombres, J., \textit{Functional equations in several variables with applications to mathematics, Information theory and to the natural and social sciences.} Encyclopedia of Mathematics and its Applications, vol. 31, Cambridge University Press, Cambridge (34B40) MR1004465 (90h:39001) (1989).
\bibitem{Ajb} Ajebbar, O., Elqorachi, E., \textit{The Cosine-Sine functional equation on a semigroup with an involutive automorphism.} Aequat Math, 91, 1115--1146 (2017).
\bibitem{Ajb2} Ajebbar, O., Elqorachi, E., \textit{Solutions and stability of trigonometric functional equations on an amenable group with an involutive automorphism.} Commun Korean Math, 34(1), 55–82 (2019). https://doi.org/10.4134/CKMS.C170487
\bibitem{Ase} Aserrar, Y., Elqorachi, E., \textit{A d'Alembert type functional equation on semigroups.}
(Manuscript 2022).
\bibitem{Ch} Chung, J.K, Kannappan, Pl., Ng, C.T., \textit{A generalization of the Cosine-Sine functional equation on groups.} Linear Algebra Appl, 66, 259--277 (1985).
\bibitem{EB1}  Ebanks, B., \textit{The cosine and sine addition and subtraction formulas on semigroups.} Acta Math. Hungar. 165, 337–354 (2021). https://doi.org/10.1007/s10474-021-01167-1
\bibitem{EB}  Ebanks, B., \textit{The sine addition and subtraction formulas on semigroups.}
Acta Math. Hungar, 164 (2) 533--555 (2021).
\bibitem{EB2}  Ebanks, B., \textit{Around the Sine Addition Law and
d’Alembert’s Equation on Semigroups.}
Results Math.  77, 11 (2022). https://doi.org/10.1007/s00025-021-01548-6
\bibitem{Ebanks} Ebanks, B., \textit{The Cosine-Sine Functional Equation on Semigroups.} Ann Math. Silesianae, vol. 36, no. 1, (2022), pp. 30-52. https://doi.org/10.2478/amsil-2021-0012
\bibitem{test} Kannappan, Pl., \textit{Functional equations and inequalities with applications, Springer monographs in Mathematics, Springer, New York} (2009) https://doi.org/10.1007/987-0-387-89492-8
\bibitem{Pou} Poulsen, T.A., Stetk\ae r, H., \textit{On the trigonometric subtraction and addition formulas.} Aequat Math. 59(1-2), 84-92 (2000).
\bibitem{test} Sahoo, P.K, Kannappan, Pl, \textit{Introduction to functional equations. Chapman and Hall/ CRC Press, Boca Raton} (2011).
\bibitem{ST1} Stetk\ae r, H., \textit{Functional equations on groups.}
World scientific Publishing CO, Singapore (2013).
\bibitem{ST} Stetk\ae r, H., \textit{A variant of d'Alembert's functional equation.} Aequat Math, no 3,  89, 657--662 (2015).
\bibitem{S} Stetk\ae r, H., \textit{The cosine addition law with an additional term.} Aequat Math. 90, 1147--1168 (2016).
\bibitem{ST2} Stetk\ae r, H., \textit{A Levi-Civita functional equation on semigroups.} Aequat Math. 96, 115–127 (2022). https://doi.org/10.1007/s00010-021-00865-z
\bibitem{V} Vincze, E., \textit{Eine allgemeinere Methode in der Theorie der Funktionalgleichungen II.} Publ. Math. Debrecen, 9, 314-323 (1962).
\bibitem{Z} Zeglami, D., Tial, M., Kabbaj, S., \textit{The integral cosine addition and sine
subtraction laws.} Results Math. 73, no. 3, 97, (2018).
\end{thebibliography}
\end{document}